\documentclass[a4paper,12pt,reqno]{amsart}
\usepackage[T1]{fontenc}
\usepackage[utf8]{inputenc}
\usepackage{amssymb,bbm,mathrsfs}
\usepackage[margin=1in]{geometry}
\usepackage[inline]{enumitem}
\usepackage{graphicx,color}
\usepackage[bookmarksdepth=2]{hyperref}

\newtheorem{theorem}{Theorem}[section]
\newtheorem*{theorem*}{Theorem}
\newtheorem{lemma}[theorem]{Lemma}
\newtheorem{corollary}[theorem]{Corollary}

\theoremstyle{definition}
\newtheorem{definition}[theorem]{Definition}

\newtheorem{remark}[theorem]{Remark}

\newtheorem{assumption}{Assumption}

\theoremstyle{remark}

\newcommand{\norm}[1]{\left\lVert#1\right\rVert}

\newcommand{\D}{d}

\newcommand{\ind}{\mathbf{1}}

\newcommand{\ph}{\varphi}
\newcommand{\ro}{\varrho}

\newcommand{\abs}[1]{\left| #1 \right|}

\newcommand{\expr}[1]{\left( #1 \right)}

\newcommand{\ignore}[1]{}

\usepackage{ifthen}
\newcommand{\formula}[2][nolabel]%
{%
 \ifthenelse{\equal{#1}{nolabel}}%
 {\begin{align*} #2 \end{align*}}%
 {%
  \ifthenelse{\equal{#1}{}}%
  {\begin{align} #2 \end{align}}%
  {\begin{align} \label{#1} #2 \end{align}}%
 }%
}

\newcommand{\mk}[1]{{#1}}

\title[Decay rate of harmonic functions for non-symmetric
	$\alpha$-stable processes]{\mk{D}ecay rate of harmonic functions for non-symmetric
	strictly $\alpha$-stable Lévy processes}
\author{Tomasz Juszczyszyn}
\begin{document}

\begin{abstract}
	In this paper we investigate functions that are harmonic with respect to the non-symmetric strictly $\alpha$-stable L\'evy processes on an open set $D \in \mathbb{R}^d$. We obtain the explicit formula for their boundary decay rate at parts of the boudary of $D$ outside of which they vanish.
\end{abstract}

\maketitle

\section{Introduction}
\label{sec:intro}
With rare exceptions, explict boundary decay rate of harmonic functions for jump Markov type processes, or non-local operators, have been studied under the symmetry assumption. The only result for non-symmetric processes or operators known to the author are (\cite{bib:fr16}) and (\cite{bib:cw17}). Here we provide the boundary decay rate for functions harmonic with respect to general stable L\'evy process, an important class of Markov processes with numerous applications. Our result requires realatively mild assumptions on the jump kernel, and works for sufficiently smooth sets.

The important tool in investigating the behaviour of \mk{harmonic} functions near the boundary or existence of their limits is the boundary Harnack inequality. It is a statement about positive harmonic functions in an open set $D$, which are equal to zero on a part of the boundary. It states that if $D$ is regular enough (for example, a Lipschitz domain), $z$ is a boundary point of $D$, $f$ and $g$ are positive and harmonic in $D$, and both $f$ and $g$ converge~to~$0$ on $\partial D \cap B(z, R)$, then for every $r \in (0, R)$
\formula[eq:BHI]{
 \sup_{x \in D \cap B(z, r)} \frac{f(x)}{g(x)} & \le c_{BHI} \inf_{x \in D \cap B(z, r)} \frac{f(x)}{g(x)} \, ,
}
where constant $c_{BHI}$ does not depend on $f$ and $g$.

\smallskip

BHI for harmonic functions of the Laplacian $\Delta$ in Lipschitz domains was proved in 1977–78 by B. Dahlberg (\cite{bib:d77}), A. Ancona (\cite{bib:a78}) and J.-M. Wu (\cite{bib:w78}). In 1989 R. Bass and K. Burdzy proposed an alternative probabilistic proof based on elementary properties of the Brownian motion (\cite{bib:bb90}). 

\smallskip

It is possible to define harmonicity in more probabilistic terms. Let $X$ be a Brownian motion and let $P_t$ be its transition semi-group defined by
\formula{
	P_t f (x) = {\mathbb{E}^x}f(X_t).
}
Then the generator of $P_t$ is Laplacian $\Delta$. Moreover every function $f$ is harmonic in an open set $D$ if and only if for any $x \in B, \overline{B} \subset D$ we have
\formula{
	f(x) = \mathbb{E}^x f(X_{\tau_B}\mathbbm{1}_{\{\tau_B < \infty\}}), && x \in D,
}
where $\tau_B$ is first exit time of $X$ from $B$.

It is possible to extend the definition of Laplacian and corresponding harmonic functions on to non-local operators  by changing the underlying stochastic process.
\smallskip

In 1997 K. Bogdan proved BHI for the fractional Laplacian $\Delta^{\alpha/2}$ (and the isotropic $\alpha$-stable L\'evy process for $0<\alpha<2$) and Lipschitz sets (\hspace{1sp}\cite{bib:b97}). In 1999 R. Song and J. -M. Wu extended the result to all open sets (\hspace{1sp}\cite{bib:sw99}) with $c_{BHI}$ depending on $d,D,z,r$, and in 2007 K. Bogdan, T. Kulczycki and M. Kwa\'snicki extended their result (\hspace{1sp}\cite{bib:bkk08}) by showing that $c_{BHI}$ in fact only depends on $\alpha$ and $d$. In 2008 P. Kim, R. Song and Z. Vondra\v{c}ek proved BHI for subordinate Brownian motions in "fat" sets and in 2011 extended it to a more general class of isotropic L\'evy processes and arbitrary domains (\hspace{1sp}\cite{bib:ksv09},\cite{bib:ksv12}). In 2014 K. Bogdan, T. Kumagai and M. Kwa\'snicki proved BHI for a wide class of non-symmetric processes in duality (\hspace{1sp}\cite{bib:bkk15}). In 2016 similar result was obtained by Z.-Q. Chen, Y.-X. Ren and T. Yang for $\kappa$-fat sets and some processes without dual process (\hspace{1sp}\cite{bib:cry16}). Finally, in 2016 X. Ros-Oton and J. Serra proved BHI for arbitrary open sets and operators with kernels, which are comparable with stable kernels (\hspace{1sp}\cite{bib:rs16}).

\smallskip

In most of the cases mentioned above the constant $c_{BHI}$ in \eqref{eq:BHI} converges to $1$ as $r \to 0$ giving the existence of boundary limits of ratios of harmonic functions. Methods used in those proofs involve so-called \emph{reduction of oscillation}. For jump-type processes this requires additional assumptions (scale invariance of BHI or uniformity of BHI). One of the results, which we will refer to in this paper, is found independently by M. Kwaśnicki and the author (\hspace{1sp}\cite{bib:jk16}), and by P. Kim, R. Song and Z. Vondraček in (\hspace{1sp}\cite{bib:ksv16}), where the existence of the limits is proven for a wide class of non-symmetric processes and arbitrary open sets. 

\smallskip

A natural consequence of the existence of limits of ratios of harmonic functions is the question about explicit decay rate of such functions near the boundary of $D$. The answer is known for a wide class of symmetric processes. For example, P. Kim, R. Song and Z. Vondra\v{c}ek proved in 2014 the result for subordinate Brownian motions where the Laplace exponent $\phi$ of the subordinator satisfies mild scaling conditions (\cite{bib:ksv14}). It states that if $X$ is subordinate Brownian motion then for every $C^{1,1}$ set $D$, every $r>0$, $z \in \partial D$, and every non-negative function $u$ in $\mathbb{R}^d$ which is harmonic in $D \cap B(z,r)$ with respect to $X$ and vanishing continuously on $D^c$ the limit
\formula{
\lim_{x \to z}\frac{u(x)}{\sqrt {\phi (\delta_D (x)^{-2})}} 
}
exists. Another result is the work of T. Grzywny, K.-Y. Kim and P. Kim from 2015, who obtained the decay rates for a large class of symmetric pure jump Markov processes dominated by isotropic unimodal L\'evy processes with weak scaling conditions for sets of class $C^{1,\ro}$ for $\ro \in (\alpha/2,1]$ (\hspace{1sp}\cite{bib:gkk15}).

\smallskip

To our knowledge not much is known about decay rates in the non-symmetric case. Here we would like to mention the work of X. Fernández-Real and X. Ros-Oton for symmetric $\alpha$-stable process with drift (\hspace{1sp}\cite{bib:fr16}) and an ongoing work of Z.-Q. Chen and L. Wang (\hspace{1sp}\cite{bib:cw17}). The goal of this article is to obtain explicit decay rate of harmonic functions in sufficiently regular sets for non-symmetric, strictly $\alpha$-stable processes. The following is our main result.
\begin{theorem}
\label{thm:1}
Let $X$ be a (possibly non-symmetric) $\mathbb{R}^d$-valued strictly $\alpha$-stable process with $\alpha~\in~(0,2)$ and the L\'evy measure given by formula
\formula{
\nu(A) = \int_{A} \frac{1}{\abs{x}^{d+\alpha}}\vartheta\expr{\frac{x}{\abs{x}}}dx,
}
where $\vartheta $ is strictly posivite and of class $ C^{\epsilon}$ on the unit sphere for some $\epsilon > 0$. Let $D$ be a bounded, open $C^{1,1}$ set if $\alpha < 1$, and $C^{2,\alpha -1+\epsilon}$ if $\alpha \geq 1$. Let $z \in \partial D$. Then for every non-negative function $f$, harmonic in $D \cap B(z,R_0)$ with respect to the process $X$ and vanishing continuously on $D^c \cap B(z,R_0)$ the limit
	\formula{
		\lim_{\substack{x \to z \\ x \in D}} \frac{f(x)} {\abs{x-x_D}^{\beta(x_D)}}
	}
exists, where $x_D \in \partial D$ is the boundary point nearest to $x$ and the exponent $\beta$ is given by formula
\formula{
\beta(x) = \alpha \mathbb{P}^0(\langle X_t, x - x_D \rangle > 0)
}

\end{theorem}

\section{Preliminaries}
\label{sec:preliminaries}

\subsection{Notation, definitions and technical lemmas}

In this section we define objects and properties that will appear in further parts of this paper, discuss our assumptions, as well as some technical lemmas used in the proof.

\smallskip
Let us fix the notation. Throughout this paper $d \geq 2$. By $\langle \cdot,\cdot\rangle$ we denote the usual dot product in $\mathbb{R}^d$. We denote the Euclidean distance between $x$ and $y$ by $\abs{x-y}$ and the Euclidean distance between $x$ and $D^c$ \mk{by} $\delta_{D} (x)$. With $:=$ we define new objects. Each constant, unless stated otherwise, is positive. By $c$ and $c_i, i \in \mathbb{N}$ we denote constants that are less important, thus they may represent different values even in the scope of one lemma or theorem. By $c(a)$ we denote that constant $c$ that depends on $a$. 

\smallskip
By $B(x,r)$ we denote a ball of radius $r$ with its center \mk{at} $x$. To simplify the notation we denote $D_r = D\cap B(0,r)$ and $D^{*}_r = \{x \in D: \delta_{D}(x) < r \}$. By $S^d$ we denote the unit sphere in $\mathbb{R}^d$. For $x = (x_1,x_2, \dots , x_{d}) \in \mathbb{R}^d$ we write $x =(\tilde{x},x_d)$, where $\tilde{x} = (x_1,x_2, \dots , x_{d-1})$ denotes the first $d-1$ coordinates of $x$ and $x_d$ is the last one.

\smallskip
By $C_0$ we denote the class of continuous functions on $\mathbb{R}^d$ converging to 0 as $x \to \infty$ and by $C_c$ we denote the class of functions that are compactly supported.

\smallskip
By changing the coordinate system in $\mathbb{R}^d$ we mean applying \mk{an} isometrical transformation to $\mathbb{R}^d$\mk{. Similarly, by scaling we mean an application of a dilation to $\mathbb{R}^d$. For example, by an appropriate change of the coordinate system, every open half-space with a distinguished boundary point $z$ can be transformed into $\mathbb{H} = \{x \in \mathbb{R}^d : x_d > 0\}$ in such a way that $z$ is mapped to the origin $0$. Similarly, by an appropriate change of the coordinate system and a dilation, every open half-space with a distinguished interior point $x$ can be transformed into $\mathbb{H}$ in such a way that the image of $x$ is equal to $(0, 0, \ldots, 0, 1)$.}

\begin{definition}
For a compact set $K$, we \mk{write} $f \in C^{n,\gamma}(K)$ if the $n$-th order partial derivatives of $f$ are H\"older continuous on $K$ with exponent $\gamma$ ($0 < \gamma \leq 1$). Such functions form a Banach space with norm
\formula{
	\norm{f}_{C^{n,\gamma}(K)} =
	\begin{cases}
		\norm{f}_{L^{\infty}(K)} + \sup \left\{ \frac{\abs{f(y)-f(x)}}{\abs{y-x}^{\gamma}}: x,y \in K  \right\}  & \text{ if } n=0\mk{,}\\
		\Sigma_{\abs{\nu}<n} \norm{D^{\nu}f}_{C^{0,\gamma}(K)}& \text{ if } n \neq 0 \mk{.}
	\end{cases}
}
	For simplicity we write $C^{\kappa} := C^{n,\gamma}$, where $n =  \lfloor \kappa \rfloor, \gamma = \kappa - \lfloor \kappa \rfloor$ when $\kappa > 0$ is not an integer.
\end{definition}

\begin{definition}
	For an open set $D$, we \mk{write} $f \in C^{n,\gamma}(D)$ if $f \in C^{n,\gamma}(K)$ for every compact subset $K$ of $D$.
\end{definition}

\begin{definition}
\label{def:c11}
An open set $D$ in $\mathbb{R}^d$ is of class $C^{n,\gamma}$ if there exists a radius $r>0$ and a constant $C$ such that for every $z \in \partial D$ there exist an isometry $\phi: \mathbb{R}^d \rightarrow \mathbb{R}^d$  and a function $f \in C^{n,\gamma}(\mathbb{R}^{d-1})$ such that $\phi(z) = 0$, $\norm{f}_{C^{n,\gamma}(\mathbb{R}^{d-1})} \leq C$ and $\phi(D) \cap B(0,r) = \{x \in \mathbb{R}^d: x_d > f(\tilde{x})\} \cap B(0,r)$.
\end{definition}

\mk{Recall} that a random variable $X$ has a \emph{strictly stable distribution} if for every $a,b > 0$ there exist $c >0$ such that $aX_1 + bX_2$ and $cX$ have the same distribution if $X_1$, $X_2$ are independent copies of $X$. In this case there exists $\alpha \in (0,2]$ such that $a^{\alpha} + b^{\alpha} = c^{\alpha}$.
We say that $\alpha$ is the \emph{index of stability} of $X$.

\mk{Recall also} that $X = \{ X_t\}_{t \in \mk{[}0,\infty)}$ is a L\'evy process if it is an $\mathbb{R}^d$-valued stochastic process with $X_0 = 0$, stationary and independent increments and c\`adl\`ag paths. 

A L\'evy process is described by the characteristic exponent $\Psi$, which is given by the L\'evy--Khintchine formula:
\formula[eq:levyexp]{
\Psi(\xi) & = \log(\mathbb{E} e^{i \langle \xi, X_1 \rangle}) = -a \abs{\xi}^2 + i \langle \gamma, \xi \rangle - \int_{\mathbb{R}^d \setminus \{0\}} (1 - e^{i \langle \xi, z \rangle} + i \langle \xi, z \rangle \mathbbm{1}_{B(0,1)}(z)) \nu(\D z)
}
for $\xi \in \mathbb{R}^d$, where $a \ge 0$ is the Gaussian component, $\gamma \in \mathbb{R}^d$ is the drift coefficient and $\nu$ is a non-negative measure such that $\int_{\mathbb{R}^d\setminus \{0\}} \min(1, \abs{z}^2) \nu(dz) < \infty$, called L\'evy measure. By $\mathbb{E}^x$ we denote the expectation corresponding to the process $X_t$ with condition $X_0 = 0$ a.s replaced by $X_0 = x$ a.s.
\smallskip
By $\tau_D$ we denote the first time the process $X$ exits an open set $D$, that is
\formula{
	\tau_D = \inf\{ t>0: X_t \notin D \}.
}

We say that  $X = \{ X_t\}_{t \in (0,\infty)}$ is a strictly $\alpha$-stable L\'evy process when it is a L\'evy process such that $X_t$ has strictly $\alpha$-stable distribution for every $t>0$.

%\begin{definition}
%	We define the \emph{transition density function} $p_t(x)$ of a L\'evy process $X$ as the density function of the distribution of process $X$ at time $t$. We define $p_t(x,y) = p_t(y-x)$.
%\end{definition}

%\begin{definition}
%We define the \emph{Green function} $G_D(x,y)$ of a L\'evy process $X$ and open set $D$ by formula
%\formula{
%G_D(x,y) = \int_0^{\tau_D}p_t(x,y)dt,
%}
%where $x,y \in D$
%\end{definition}

\begin{definition}
We define \emph{transition operator} $p_t$ of the process $X$ by the formula
\formula{
	p_tf(x) = \mathbb{E}^x f(X_t)
} 
and the \emph{generator} $\mathcal{L}$ of the process $X$ applied to \mk{a} function $f$ by the formula 
\formula[eq:gen]{
\mathcal{L}f(x) = \lim_{t \to 0^+} \frac{p_tf(x) - f(x)}{t}
}
for every $f \in C_0$ such that above limit exists uniformly on $\mathbb{R}^d$.
\end{definition}

\begin{definition}
	We define the \emph{Dynkin generator} $\mathcal{L_D}$ of the process $X$ applied to \mk{a} function $f$ at \mk{a} point $x$ by the formula 
	\formula[eq:gen_dynkin]{
		\mathcal{L_D}f(x) = \lim_{r \to 0^+} \frac{\mathbb{E}^xf(X_{\tau_{B(x,r)}}) - f(x)}{\mathbb{E}^x\tau_{B(x,r)}}
	}
	for every $f \in C_0$ and $x \in \mathbb{R}^d$ such that above limit exists.
\end{definition}

\mk{It is known that if $f$ is in the domain of the generator $\mathcal{L}$, then $\mathcal{L_D} f(x)$ is well-defined for every $x$ and $\mathcal{L_D} f(x) = \mathcal{L} f(x)$. Conversely, if $f \in C_0$, $\mathcal{L_D} f(x)$ is well-defined for every $x$, and $\mathcal{L_D} f \in C_0$, then $f$ is in the domain of $\mathcal{L}$. We refer to Chapter~V in~\cite{dynkin} for a proof and further discussion.}

For every open set $D$ there exists a \emph{Green function} $G_D(x,y)$ such that $G_D(x,y) \geq 0$ for $x,y \in D$ and $G_D(x,y) = 0$ for $x \in D^c$ or $y \in D^c$ such that \mk{$G_D(x, y)$ is a continuous map from $D \times D$ into $[0, \infty]$, and
\formula{
\int_D G_D(x,y) f(y) dy = \mathbb{E}^x \int_0^{\tau_D} f(X_t) dt
}
for every non-negative function $f$. In particular,}
\formula{
\int_D G_D(x,y) dy = \mathbb{E}^x \tau_D
}
for every $x,y \in D$.

In further parts of this article we \mk{will} use \mk{the} \emph{Ikeda--Watanabe formula}. It states that for every open set $D$, $x \in D$ and \mk{a} L\'evy process $X$ with L\'evy measure $\nu$ we have:  
\formula{
\mathbb{E}^x(f(X_{\tau_D})) = \int_D G_D(x,y)\int_{D^c} \nu(z-y)f(z)dzdy
}
for every \mk{non-negative} function $f$ such that \mk{$f = 0$ in $\overline{D}$}.

\begin{definition} 
We say that a function $f$ is \emph{harmonic} for $X$ in an open set $D$ if for every bounded open set $B$ such that $\overline{B} \subset D$ and $x \in B$ we have 
\formula{
	\mathbb{E}^xf(X_{\tau_B}) = f(x). 
}
We say that a function is \emph{regular harmonic} whenthe above equality holds also for $B = D$. If a function is regular harmonic in an open set $D$, then it is regular harmonic in any open subset of $D$.
\end{definition}

\begin{remark}
	\label{rem:dynkin_harmonic}
	If a function $f$ is harmonic in an open set $D$, for every $x$ in $D$ we have $\mathcal{L_D}f(x) = 0$.
\end{remark}

We proceed with two elementary, technical results.

\begin{lemma}
For $p, q \in (0,1)$, $x,y > 0$ and $\eta \in (0,1]$ there exists $c = c(q,\eta)$ such that 
\formula[eq:oszacowanie1]{
\abs{x^q - y^q} \leq c \max(x,y)^{q - \eta}\abs{x - y}^{\eta},
}
\formula[eq:oszacowanie2]{
\abs{x^p - x^q} \leq \abs{\ln(x)}\max(x^p,x^q)\abs{p - q}.
}
\end{lemma}

\begin{proof}
Without loss of generality we assume that $x > y$. We have
\formula{
\abs{x^q - y^q} &= \frac{\abs{x^q - y^q}}{\abs{x - y}^{\eta}} \abs{x - y}^{\eta} = \frac{\abs{1 - (\frac{y}{x})^q}}{\abs{1 - \frac{y}{x}}^{\eta}} x^{q - {\eta}} \abs{x - y}^{\eta} = \frac{\abs{1 - s^q}}{\abs{1 - s}^{\eta}} x^{q - {\eta}} \abs{x - y}^{\eta}
}
for $s = \frac{y}{x} \in [0,1)$. Since $0 < {\eta} \leq 1$, by l'Hospital's rule,
\formula{
\lim_{s \to 1}\frac{1 - s^q}{(1 - s)^{\eta}} = \lim_{s \to 1}\frac{q s^{q-1}}{{\eta}(1 - s)^{{\eta}-1}} 
}
is equal to $0$ for ${\eta} < 1$ and $q$ for ${\eta} = 1$. Since the function $\frac{1 - s^q}{(1 - s)^{\eta}}$ is continuous on $[0,1)$ and has a limit as $s \to 1$, it is bounded on $[0,1]$ by some constant $c$. It follows that
\formula{
\abs{x^q - y^q} \leq c \max(x,y)^{q - {\eta}}\abs{x - y}^{\eta}.
}
For the second inequality we write
\formula{
		\abs{x^p - x^q} = \abs{\int_p^q \ln(x)x^t dt} \leq \abs{\ln(x)}\max(x^p,x^q)\abs{p - q}.
}
\end{proof}

\begin{lemma}
\label{lem:lagrange}
For any  closed, convex set $K$ and any function $f \in C^{\gamma}(K)$, $1 < \gamma < 2$, we have
\formula{
\abs{f(x) - f(y) - \langle x-y,\nabla f(y) \rangle} \leq \|f\|_{C^{\gamma}(K)}\abs{x-y}^{\gamma}
}
for every $x,y \in K$.
\end{lemma}

\begin{proof}
	By the mean value theorem,
	\formula{
		f(x) - f(y) - \langle x-y,\nabla f(y) \rangle &= \langle x-y,\nabla f(x_1)\rangle - \langle x-y,\nabla f(y) \rangle \\
		&= \langle x-y,\nabla f(x_1) - \nabla f(y)\rangle
	}
	for some $x_1 = (1-s)x+sy, s\in [0,1]$. Thus we get that
	\formula{
		\abs{f(x) - f(y) - \langle x-y,\nabla f(y)\rangle} & \leq \|f\|_{C^{\gamma}(K)} \abs{x-y} \abs{x_1-y}^{\gamma-1} \leq \|f\|_{C^{\gamma}(K)}\abs{x-y}^{\gamma}. \qedhere
	}
\end{proof}

\smallskip
\subsection{Assumptions and properties of the process $X$}

\begin{assumption}
\label{as:X} 
We assume that $X$ is a strictly $\alpha$-stable $d$-\mk{dimensional} L\'evy process with $d \geq 2,\alpha \in (0,2)$. We assume that \mk{the} L\'evy measure \mk{of $X$} is absolutely continuous with respect to the Lebesgue measure and \mk{it} is given by formula
\formula{
\nu(A) = \int_{A} \frac{1}{\abs{x}^{d+\alpha}} \vartheta \expr{\frac{z}{\abs{z}}}dz,
}
where $\vartheta \in C^{\epsilon}(S)$ for some $\epsilon>0$ and $\vartheta(z) > 0$ for all $z \in S$. 
\end{assumption}

Assumption \ref{as:X} implies that if $\alpha \neq 1$ the L\'evy--Khintchine exponent of the process $X$ has coefficients $a$ and $\gamma$ equal to $0$. If $\alpha = 1$\mk{, the} coefficient $a$ is equal to $0$ and the function $\vartheta$ is symmetric. Moreover\mk{,} strictly $\alpha$-stable processes are scaling invariant.

\begin{definition}
We define the \emph{pointwise generator} $\mathcal{A}$ of process $X$ at point $x$ by formula
\formula[eq:generator]{
\begin{aligned}
	&\mathcal{A}f(x) =  \int_{\mathbb{R}^d} (f(y) - f(x)) \nu (y-x)dy & \text{ if } \alpha < 1\\
	&\mathcal{A}f(x) = \langle \gamma, \nabla f(x) \rangle + \int_{\mathbb{R}^d} (f(y) - f(x) - \langle\nabla f (x),y-x\rangle\mathbbm{1}_{B(x,r)}(y)) \nu (y-x)dy &\text{ if }\alpha = 1\\
	&\mathcal{A}f(x) = \int_{\mathbb{R}^d} (f(y) - f(x) - \langle\nabla f (x),y-x\rangle ) \nu (y-x)dy &\text{ if }\alpha > 1
\end{aligned}
}
for every function $f$ for which the integral is finite at $x$. In particular\mk{,} this is the case for any bounded function $f$ \mk{which is $C^{\alpha + \epsilon}$ in} some neighbourhood of $x$ \mk{for some $\epsilon > 0$; see} \cite{bib:ks19}. Note that in case $\alpha = 1$, since the L\'evy measure of the process is symmetric, the definition of $\mathcal{A}$ does not depend on $r >0$.
\end{definition}

\begin{definition}
For any unit vector $u \in S^d$, we define the one-dimensional L\'evy process $X^{u} = \{\langle X_t,u \rangle\}_{t \in \mathbb{R}^+}$ which is the orthogonal projection of $X$ onto the line $\{x u: x \in \mathbb{R}\}$. By $\nu_{u}$ we denote its L\'evy measure. Also, for $z \in \mathbb{R}^d$ by $\mathbb{H}_{u,z}$ we denote the half-space $\{x: \langle x-z,u \rangle > 0 \}$.
\end{definition}

\begin{lemma}
The process $X^u$ is a one-dimensional strictly $\alpha$-stable L\'evy process. Its L\'evy measure $\nu_{u}$ is absolutely continuous with respect to the Lebesgue measure and its density $\nu_{u}(z)$ is given by the formula
\formula[eq:dens_nu_u]{
\begin{aligned}
\nu_{u}(z) &= \frac{1}{z^{\alpha + 1}} \int_{S_{u}} \vartheta(w) \langle u,w \rangle ^{\alpha} dw && \text{ if }z > 0 \\
\nu_{u}(z) &= \nu_{-u}(-z) && \text{ if }z < 0,
\end{aligned}
}
where $S_{u} = S^d \cap \mathbb{H}_{u,0}$ and $dw$ is the surface measure on the unit sphere.
\end{lemma} 
 
\begin{proof}
We begin by calculating the tail of the measure $\nu_u$. Let $x \in \mathbb{R}^d$ and let $z_0 >0$. We have
\formula{
\int_{z_0}^{\infty} \nu_u(dz) &= \int_{\mathbb{H}_{u,z_0u}} \abs{x}^{-d-\alpha} \vartheta \expr{\frac{x}{\abs{x}}} dx.
}
We use spherical coordinates:
\formula{
\int_{z_0}^{\infty} \nu_u(dz) &= \int_{S_{u}} \int_{z_0/\langle u,w \rangle}^{\infty}r^{-1-\alpha} \vartheta(w)  drdw = \frac{1}{\alpha}\frac{1}{z_0^\alpha}\int_{S_{u}} \vartheta(z)\langle u,w \rangle^{\alpha}dw.
}
By differentiation, we get \eqref{eq:dens_nu_u}. The case of $z_0 < 0$ is very similar. 
\end{proof}

Since $X^u$ is a one-dimensional $\alpha$-stable L\'evy process\mk{,} below we recall some facts about harmonic functions for these processes. 

\begin{theorem}[see Example 2 in \cite{bib:b99}]
\label{thm:one-dim}
Let $Y$ be a one-dimensional $\alpha$-stable L\'evy process. Let $\beta = \alpha\mathbb{P}(Y_1 > 0)$. Then the function
\formula{
h(x) &= x^{\beta}\mathbbm{1}_{(0,\infty)}(x)
} 
is regular harmonic for $Y$ in $(0,a)$ for every $a>0$. 
\end{theorem}

Recall that for a one-dimensional strictly $\alpha$-stable L\'evy process, the L\'evy measure $\mu$ is absolutely continuous with respect to the Lebesgue measure and its density is given by formula
\formula[eq:one_dim_dens]{
\mu(z) = C^- \frac{1}{\abs{z}^{\alpha + 1}} \ind_{(-\infty,0)}(z) + C^+ \frac{1}{\abs{z}^{\alpha + 1}}(z) \ind_{(0,\infty)},
}
where $C^-, C^+ \geq 0, C^- + C^+ > 0$ and if $\alpha = 1$ then \mk{necessarily} $C^- = C^+$.
In that case the parameter $\beta$ can be given explicitly by the formula (see \cite{bib:z57})
\formula[eq:alpha_plus]{
	\begin{aligned}
		\beta &= \frac{\alpha}{2} + \frac{1}{\pi} \arctan \expr{\frac{C^+ - C^-}{C^+ + C^-} \tan \expr{\frac{\alpha \pi}{2}}}\\
	\end{aligned}
}
if $\alpha \neq 1$, while for $\alpha = 1$ we have $C^+ = C^- > 0$ and
\formula[eq:alpha_plus1]{
	\beta &= \mathbb{P}(X_1 > 0) = \int_{0}^{\infty} \frac{1}{\pi} \frac{C^+}{(C^+)^2+(x-b)^2}dx = \frac{1}{2} + \frac{1}{\pi} \arctan(b/C^+),
}
where $b$ is the the drift of the process.

In the remaining part of this article we \mk{use} the objects defined in Theorem~\ref{thm:one-dim} for projections $X^u$. In this case we denote the dependence on $u$ by writing $C^+(u), C^-(u)$ and $\beta(u)$.

\begin{lemma}
	\label{lem:beta_ogr}
	There are constants $\beta_{min}(X) > \max\{0, \alpha -1\}$ and $\beta_{max}(X) < \min\{\alpha,1\}$ such that
	\formula{
		\beta_{min}(X) \leq \beta(u) \leq \beta_{max}(X)
	} 
	for every $u \in S^d$.
\end{lemma}

\begin{proof}
	There exists a constant $c(X) \in (0,1)$ such that 
	\formula{
		-1 + c \leq \frac{C^+(u) - C^-(u)}{C^+(u) + C^-(u)} \leq 1-c
	} 
	for any $u \in S$.
	By \eqref{eq:alpha_plus} we have
\formula{
	\frac{\alpha}{2} - \frac{1}{\pi}\abs{\arctan\expr{\tan\frac{\alpha \pi}{2}}} < \beta(u) < \frac{\alpha}{2} + \frac{1}{\pi}\abs{\arctan\expr{\tan\frac{\alpha \pi}{2}}}
} 
for $\alpha \in (0,2) \setminus \{1\}$. Since 
\formula{
	\arctan\expr{\tan\frac{\alpha \pi}{2}} =
	\begin{cases}
		\frac{\alpha \pi}{2}  & \text{ if } \alpha < 1 \mk{,} \\
		\frac{(\alpha - 2) \pi}{2} & \text{ if } \alpha > 1 \mk{,}
	\end{cases}
}
and since $\beta(u)$ is \mk{a} continuous function on a compact set\mk{,} we have
\formula{
\max\{0, \alpha -1\} < \beta(u) < \min\{\alpha,1\}
}
for $\alpha \in (0,2) \setminus \{1\}$.
\mk{When} $\alpha = 1$\mk{, the desired result} follows from \eqref{eq:alpha_plus1}.
\end{proof}

Note that the constants $\beta_{min}$ and $\beta_{max}$, even though depend\mk{ent} on $X$, \mk{are invariant under a} change the coordinate system \mk{and scaling}. We keep the notation $\beta_{min}$ and $\beta_{max}$ till the end of this article.

\begin{remark}
	\label{rem:beta_holder}
	By \eqref{eq:one_dim_dens} and \eqref{eq:dens_nu_u}\mk{, the function} $C^+(u)$ \mk{is a spherical} convolution of \mk{a $C^{\alpha}(S^d)$ `zonal' function $w \mapsto (\max\{\langle u, w\rangle, 0\})^\alpha$} and \mk{a} $C^{\epsilon}(S^d)$ function \mk{$\theta$. Thus, $C^+(u)$ and $C^-(u) = C^+(-u)$ belong to} $C^{\alpha+\epsilon}(S^d)$. By \eqref{eq:alpha_plus} for $\alpha \neq 1$\mk{,} and \mk{by} \eqref{eq:alpha_plus1} for $\alpha = 1$\mk{,} we have that $\beta(u)$ is \mk{in} $C^{\alpha+\epsilon}(S^d)$.
\end{remark}

\begin{lemma}
\label{tw:harmonicznosc_w_R^d}
Let $X$ be a $d$-dimensional strictly $\alpha$-stable L\'evy process. Let $u \in S^d$. Then the function
\formula{
h_{u,z}(x) = (\delta_{\mathbb{H}_{u,z}}(x))^{\beta(u)}
}
is a regular harmonic function for $X$ in $D \cap \mathbb{H}_{u,z}$ for every bounded open set $D$.
\end{lemma}

\begin{proof}
\mk{By an appropriate} change the coordinate system \mk{and scaling, we may assume} that $z=0$ and $u=(0,...,0,1)$. Let $D$ be a bounded open set. Let $h(x_d)$ be defined as in Theorem \ref{thm:one-dim} for $Y = X^u$. We define $U_t =\{z \in \mathbb{R}^d: 0 < x_d < t \}$ and \mk{we choose} $t$ such that $D \cap \mathbb{H}_{u,z} \subset U_t$. We have
\formula[eq:harm_proj]{
\mathbb{E}^x h_{u,z}(X_{\tau_{\mk{U}_t}}) = \mathbb{E}^{x_d} h(X^u_{\tau_{(0,t)}}) = h(x_d) = h_{u,z}(x),
} 
where $\mathbb{E}^x$ and $\mathbb{E}^{x_d}$ are the expectations for the $d$-dimensional process \mk{$X$} and its orthogonal projection \mk{$X^u$,} respectively. By \eqref{eq:harm_proj}\mk{,} the function $h_{u,z}$ is regular harmonic in $\mk{U}_t$. Since $D \cap \mathbb{H}_{u,z} \subset \mk{U}_t$, $h_{u,z}$ is also regular harmonic in $D \cap \mathbb{H}_{u,z}$.  
\end{proof}

\begin{corollary}
\label{cor:dynkin_pionwise}
\mk{The} function $\mathcal{A}h_{u,z}(x)$ is well defined for every $x \in \mathbb{H}_{u,z}$ and 
\formula{
	\mathcal{A}h_{u,z}(x) = 0
}
for $x \in \mathbb{H}_{u,z}$.
\end{corollary}

\begin{proof}
Let $x$ be a point in $\mathbb{H}_{u,z}$ and \mk{let} $r$ be a radius such that $B(x,r) \subset \mathbb{H}_{u,z}$. \mk{The} function $h_{u,z}$ belongs to $C^{\infty}(B(x,r))$. Since it is harmonic in $\mathbb{H}_{u,z}$, by Remark \ref{rem:dynkin_harmonic} it belongs to the domain of the Dynkin generator $\mathcal{L_D}$ at the point $x$ and $\mathcal{L_D}h_{u,z}(x)=0$. We define \mk{the} function $h^*_{u,z}$ by \mk{the} formula
\formula{
	h^*_{u,z}(y) = h_{u,z}(y) 
}
for $y \in B(x,r)$ and extend it to a smooth and compactly supported function. \mk{The function} $h^*_{u,z}$ also belongs to the \mk{domain} of \mk{the} Dynkin generator $\mathcal{L_D}$, \mk{as well as to} the domain of \mk{the} pointwise \mk{generator} $\mathcal{A}$, and
\formula[eq:dynkin_pointwise1]{
\mathcal{L_D}h^*_{u,z}(x) = \mathcal{A}h^*_{u,z}(x).
} 
The difference $h_{u,z}(x) - h^*_{u,z}(x)$ is equal to $0$ on $B(x,r)$, thus, by \mk{the} Ikeda--Watanabe formula, we have
\formula{
\mathcal{L_D}(h_{u,z} - h^*_{u,z})(x) = \lim_{s \to 0^+} \frac{\int_{B(x,s)} G_{B(x,s)}(x,y)\int_{B(x,r)^c} \nu(\mk{v}-y)(h_{u,z} - h^*_{u,z})(\mk{v})d\mk{v}dy}{\int_{B(x,s)} G_{B(x,s)}(x,y)dy}.
}
where \mk{$G_{B(x,s)}(x,z)$} is the Green function of $B(x,s)$.
\mk{Observe that if $y \in B(x, r/2)$ and $v \in B(x, r)^c$, we have $\abs{v-y}>r/2$}, and hence $\nu(\mk{v}-y)(h_{u,z} - h^*_{u,z})$ \mk{is a continuous function of $y \in B(x, r/2)$ and $v \in B(x, r)^c$, bounded by an integrable function of $v \in B(x, r)^c$ uniformly with respect to $y \in B(x, r/2)$. It follows} that $\int_{B(x,r)^c} \nu(z-y)(h_{u,z} - h^*_{u,z})(z)dz$ is a continuous function of \mk{$y \in B(x, r/2)$}. As $s \to 0$\mk{,} the measures $\frac{G_{B(x,s)}(x,y)dy}{\int_{B(x,s)} G_{B(x,s)}(x,y)dy}$ converge vaguely to the Dirac measure at the point $x$, thus we have
\formula[eq:dynkin_pointwise2]{
\mathcal{L_D}(h_{u,z} - h^*_{u,z})(x) = \int_{B(x,r)^c} \nu(z-\mk{x})(h_{u,z} - h^*_{u,z})(z)dz = \mathcal{A}(h_{u,z} - h^*_{u,z})(x).
} 
By combining \eqref{eq:dynkin_pointwise1} with \eqref{eq:dynkin_pointwise2} we get the desired result.
\end{proof}

\smallskip
\subsection{Regularity of $D$}

\begin{assumption}
\label{as:D}
If $\alpha \in (0,1)$\mk{,} we assume that $D$ is a bounded $C^{1,1}$ open set. If $\alpha \in [1,2)$\mk{,} we assume that $D$ is a bounded $C^{2,\alpha+\epsilon-1}$ set for some $\epsilon > 0$. 
\end{assumption}

\begin{remark}
\label{rem:ball_conditions}
If $D$ is a $C^{1,1}$ open set, it satisfies the uniform exterior and the uniform interior ball conditions: for some $\mk{r(D)}>0$, for every $z \in \partial D$ there are points $x_1,x_2$ such that $B(x_1,r) \subset D$, $B(x_2,r) \subset D^c$ and $z \in \overline{B}(x_1,r) \cap \overline{B}(x_2,r)$.
\end{remark}

Recall that $D^{*}_{r}=\{x \in D: \delta_D(x) < r\}$.

\begin{definition}
Let $r$ be as in Remark \ref{rem:ball_conditions}. For $x \in D^{*}_{r}$, we let $z(x)$ to be the unique point on $\partial D$ such that $\delta_{D}(x) = \abs{x - z(x)}$. If $x \in \partial D$ we define $z(x) = x$. We define $n(x)$ to be the inward-pointing normal vector to the boundary of $D$ at point $z(x)$.
\end{definition}

\begin{lemma}
\label{lem:rad_for_lip}
Let $D$ satisfy Assumption \ref{as:D}. There exists \mk{$R(D) > 0$} such that the functions $z(x)$ and $n(x)$ are Lipschitz continuous functions on $\overline{D^{*}_{R}}$ in case $\alpha < 1$ and $C^{\alpha + \epsilon}\mk{(\overline{D^*_R})}$ class functions for some $\epsilon > 0$ if $\alpha \geq 1$.
\end{lemma}

\begin{proof}
By Theorem 3.1 in \cite{bib:dz95} we get that the distance function $\delta_D(x) $ is in $C^{1,1}(\overline{D^{*}_{R}})$ for some $R$. Note that we have $\nabla \delta_D(x) = n(x)$, thus $n(x)$ is a Lipschitz function on $\overline{D^{*}_{R}}$. Since $z(x) = x - \delta_D(x) \nabla \delta_D(x)$, it is also a Lipschitz continuous function on $\overline{D^{*}_{R}}$. When $\alpha \geq 1$ and $D$ is a $C^{2,\alpha-1 + \epsilon}$ class set for some $\epsilon > 0$ then, again by Theorem 3.1 in \cite{bib:dz95}, $\delta_D(x)$ is in $C^{2,\alpha-1 + \epsilon}(\mk{\overline{D_R^*}})$, thus $\nabla \delta_D$ \mk{and} $z$ \mk{are in} $C^{\alpha + \epsilon}(\mk{\overline{D_R^*}})$.
\end{proof}

\begin{remark}
\label{rem:set_coordinate_system}
Since the \mk{harmonic functions for the} process $X$ \mk{are} scale-invariant, the constants $\beta_{max}, \beta_{min}$ and \mk{the} function $h_{u,z}$ that will be used later in this article do not change if we scale the process $X$ or (equivalently) scale \mk{the coordinate system}. To simplify \mk{the notation,} till the end of the article we \mk{choose a coordinate system, together with its scale, in such a way} that $0 \in \partial D$, 
\formula{
	e_d := n(0) = (0, \dots ,0,1) ,
}
\mk{the radius $r$ defined in Remark~\ref{rem:ball_conditions} is not less that $2$, and} the radius $R$ defined in Lemma \ref{lem:rad_for_lip} is \mk{greater than or equal to} $1$.
\end{remark}
	
\begin{corollary}
\label{cor:beta_holder}
Let $D$ satisfy Assumption \ref{as:D}. The function $\beta(n(x))$ is in $ C^{\alpha+\epsilon}(\overline{D^{*}_{1}})$ for some $\epsilon>0$ and $\norm{\beta(n(\cdot))}_{C^{\alpha+\epsilon}(\overline{D^{*}_{1}})} \leq C(X,D)$ for some $C(X,D) >0$.
\end{corollary}

\begin{proof}
The function $\beta(n(x))$ is a composition of functions $\beta$ and $n(x)$. If $\alpha<1$, by Remark~\ref{rem:beta_holder} and Lemma~\ref{lem:rad_for_lip} we have $\beta \in C^{\alpha+\epsilon}(S^d)$ for some $\epsilon > 0$ and $n(x)$ is a Lipschitz continuous function, thus their composition belongs to $C^{\alpha+\epsilon}(\overline{D^{*}_{1}})$ for some $\epsilon > 0$. If $\alpha \geq 1$, by  Remark~\ref{rem:beta_holder} and Lemma~\ref{lem:rad_for_lip} we have $\beta \in C^{\alpha+\epsilon}(S^d)$ and $n(x)\in C^{\alpha+\epsilon}(\overline{D^{*}_{1}})$, thus their composition belongs to $C^{\alpha+\epsilon}(\overline{D^{*}_{1}})$.
\end{proof}

To simplify the notation, we write $\beta(x)$ instead of $\beta(n(x))$ if $x \in \mk{\overline{D^*_1}}$.

Recall that $D_r = D \cap B(0,r)$.

\begin{lemma}
\label{lem:odl^2}
Let $D$ satisfy Assumption \ref{as:D}. For every $x = (\tilde{x},x_d) \in \overline{D}_{\mk{1}}$\mk{,} we have
	\formula[eq:odl^2]{
		\abs{\delta_D(x) - x_d} \leq \mk{\tfrac{1}{2}} \abs{\tilde{x}}^2.
	} 
\end{lemma}

\begin{proof}
Let $f$ be a function such that $\partial D \cup B(0,\mk{2})$ is contained in the graph of $f$ (see Definition \ref{def:c11}). By the uniform exterior ball condition with radius \mk{$2$,} we have
\formula{
	f(\tilde{x}) \geq - \mk{2} + \sqrt{\mk{4} - \abs{\tilde{x}}^2} \geq -\mk{\tfrac{1}{2}} \abs{\tilde{x}}^2
}
for $\abs{\tilde{x}} \le \mk{2}$. Thus\mk{,} for $x \in \mk{\overline{D_2}}$,
\formula[eq:odl^2_1]{
	\delta_D(x) \leq d(x,(\tilde{x},f(\tilde{x}))) = \abs{x_d - f(\tilde{x})} = x_d - f(\tilde{x}) \leq x_d + \tfrac{1}{2} \abs{\tilde{x}}^2.
}
On the other hand, by the fact that for $\abs{x} \le \mk{1}$ we have $\mk{2} - x_d + \mk{\tfrac{1}{2}} \abs{\tilde{x}}^2 \geq \sqrt{\abs{\tilde{x}}^2 + (\mk{2}-x_d)^2}$ (which follows by squaring both sides of the inequality) and by the uniform interior ball condition, we have
\formula[eq:odl^2_2]{
	\delta_D(x) \geq \delta_{B(\mk{2}e_d,\mk{2})}(x) = \mk{2} - \sqrt{\abs{\tilde{x}}^2 + (\mk{2}-x_d)^2} \geq x_d - \mk{\tfrac{1}{2}} \abs{\tilde{x}}^2 .
}
	By combining \eqref{eq:odl^2_1} and \eqref{eq:odl^2_2}, we get \eqref{eq:odl^2}.
\end{proof}

\section{Proof of the main theorem}
\label{sec:proof}
The main goal of this section is to provide explicit decay rates of harmonic functions at a boundary point $z$ of the set $D$. In the remaining part of the article we will always assume that the process $X$ satisfies Assumption \ref{as:X} and \mk{the} set $D$ satisfies Assumption \ref{as:D}. We choose \mk{the} coordinate system \mk{(and scaling)} as in Remark \ref{rem:set_coordinate_system} and \mk{we fix} $x_0 = (0,0, \dots ,x_{0d}) = x_{0d}e_d$ such that $\delta_D(x_0) \leq 1/2$. \mk{Finally,} we define $\mathbb{H}(x) = \mathbb{H}_{e_d,0}(x)$ and $h(x) = h_{e_d,0}(x)$.

\begin{definition}
We define the function $g$ by formula
\formula{
g(x) = (\delta_D(x))^{\beta(x)} \mathbbm{1}_{ \overline{D^*_1}}(x),
}
where the function $\beta$ is given in Corollary \ref{cor:beta_holder}\mk{; see Figure~\ref{fig:g}}.
\end{definition}

\begin{figure}
\centering
\includegraphics[width=1\textwidth]{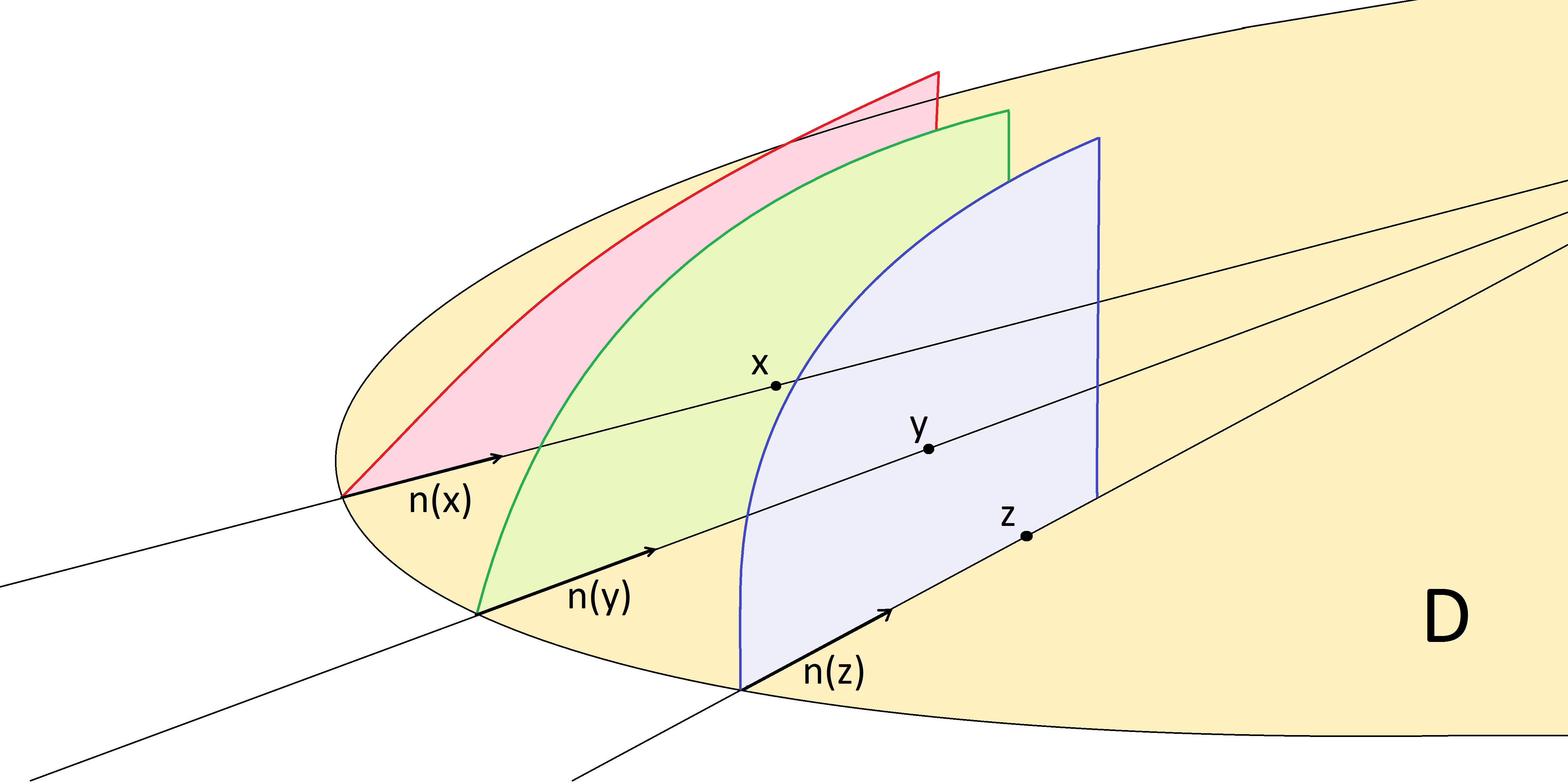}
\caption{The function $g$ is of power type with different exponents for different directions.}
\label{fig:g}
\end{figure}

\begin{remark}
\label{rem:g_Hol}
By Lemma \ref{lem:rad_for_lip} and Corollary \ref{cor:beta_holder}, $g \in C^{\alpha + \epsilon}(\overline{D^*_{1}})$ for some $\epsilon > 0$. Moreover, $g$ is bounded by $1$.
\end{remark}

\begin{remark}
\label{rem:k}
By Lemma \ref{lem:beta_ogr}, the interval $(\beta_{max},\min\{\alpha,1\})$ is non-empty. In the remaining part of this article we fix 
\formula{
\eta \in (\beta_{max},\min\{\alpha,1\}),
}
so that, by the fact that $\beta_{min} + \beta_{max} = \alpha$, we have
\formula{
2\eta - \alpha - 2 < -1\mk{,} &&\beta_{min} + \eta - \alpha > 0\mk{,} && \beta_{min} - \eta > - 1\mk{,} &&2\eta - \alpha - 1 > -1\mk{,}
}
which we will use later in this article.
\end{remark}

\begin{definition}
\label{def:split}
We define
	\formula[eq:split]{
		\begin{aligned}
			f_{1}(x) &= ((\delta_{\mathbb{H}}(x))^{\beta(x)} - (\delta_{\mathbb{H}}(x))^{\beta(x_0)})\mathbbm{1}_{D_1 \cap \mathbb{H}}(x)\mk{,} \\
			f_{2}(x) &= ((\delta_D(x))^{\beta(x)} - (\delta_{\mathbb{H}}(x))^{\beta(x)})\mathbbm{1}_{D_1 \cap \mathbb{H}}(x)\mk{,} \\
			f_3(x) &= (g(x) - h(x))\mathbbm{1}_{ (\mk{\mathbb{R}^d} \setminus (D_1 \cap \mathbb{H}))}(x).
		\end{aligned}
	}
\end{definition}

Since for $x \in D_{1} \cap \mathbb{H}$ we have
\formula{
	g(x) - h(x) &= (\delta_D(x))^{\beta(x)} - (\delta_{\mathbb{H}}(x))^{\beta(x_0)}  \\
	&= (\delta_D(x))^{\beta(x)} - (\delta_{\mathbb{H}}(x))^{\beta(x)} + (\delta_{\mathbb{H}}(x))^{\beta(x)} - (\delta_{\mathbb{H}}(x))^{\beta(x_0)},
}
by \eqref{eq:split},
\formula[eq:split2]{
g(x) - h(x) & = f_{1}(x) + f_{2}(x) + f_3(x).
}

\begin{lemma}
	\label{lem:f_1}	
	There exists $\epsilon = \epsilon(X,D) > 0$ and a constant $c = c(X,D)$ such that
	\formula{
		\abs{f_1(x)} \leq c\abs{x-x_0}^{\alpha + \epsilon}  && \text{ if } \alpha < 1\mk{,}\\
		\abs{f_1(x) - \langle \nabla f_1(x_0), x-x_0  \rangle}\leq c\abs{x-x_0}^{\alpha + \epsilon} && \text{ if } \alpha \geq 1
	}
	for $x \in D_{\mk{1}} \cap \mathbb{H}$.
\end{lemma}

\begin{proof}
	Let $x \in D_{\mk{1}} \cap \mathbb{H}$. When $\alpha < 1$, by \eqref{eq:split} \mk{and} \eqref{eq:oszacowanie2} \mk{we have
	\formula{
		\abs{f_1(x)} \leq \abs{\ln(\delta_{\mathbb{H}}(x))} \max(\delta_{\mathbb{H}}(x)^{\beta(x)},\delta_{\mathbb{H}}(x)^{\beta(x_0)}) \abs{\beta(x) - \beta(x_0)} .
	}
	Since $\beta(x), \beta(x_0) \in [\beta_{min}, \beta_{max}]$, it follows that} there exists $c(X,D)$ such that
	\formula{
		\abs{f_1(x)} \leq c\abs{\beta(x)-\beta(x_0)}\mathbbm{1}_{D_1 \cap \mathbb{H}}(x).
	} 
	By Corollary~\ref{cor:beta_holder}\mk{, $\beta$ is in $C^{\alpha + \epsilon}(\overline{D_{1}^*})$ for some $\epsilon > 0$. Thus,} we have
	\formula{
		\abs{f_1(x)} \leq c\abs{x-x_0}^{\alpha + \epsilon} \mathbbm{1}_{D_1 \cap \mathbb{H}}(x)
	}
	for some $\epsilon > 0$\mk{, as desired}.
	
	\smallskip
We now concider $\alpha \geq 1$. For the \mk{notational} convenience, till the end of this Lemma we introduce the notation: $\beta_0 := \beta(x_0), \mk{\beta_1} := \beta(x), \delta_0 := \delta_{\mathbb{H}}(x_0),\mk{\delta_1} := \delta_{\mathbb{H}}(x)$ and $v := \nabla \beta(x_0)$.
By Corollary~\ref{cor:beta_holder}, \mk{$\beta$ is in $C^{\alpha + \epsilon}(\overline{D_{1}^*})$ for some $\varepsilon > 0$. In particular, $\nabla \beta$ exists and it is a function bounded by a constant $c(X, D)$}. By a simple calculation,
	\formula{
		\nabla f_1(x_0) = (\ln \delta_0)\delta_0^{\beta_0}v .
	}
	\mk{For later needs, we record that as a consequence, for} every $\mk{\epsilon} > 0$ there exists a constant $c(X,D,\mk{\epsilon})$ such that 
	\formula[rem:nabla_f1]{
		\abs{\nabla f_1(x_0)} = \abs{\ln(\delta_0)\delta_0^{\beta_0}v} \leq c\delta_0^{\beta_0-\mk{\epsilon}} = \mk{c \delta_{\mathbb{H}}(x_0)^{\beta(x_0) - \epsilon}}.
	}
\mk{We come back to the proof of the lemma. Observe that}
\formula[eq:f1alpha>1]{
\begin{aligned}
	\abs{f_1(x) - \langle \nabla f_1(x_0), x-x_0  \rangle}&= \abs{\delta_{\mk{1}}^{\beta_{\mk{1}}}-\delta_{\mk{1}}^{\beta_0}-\langle \delta_0^{\beta_0}\ln(\delta_0)v, x-x_0 \rangle}\\ 
		&\leq  \abs{\delta_{\mk{1}}^{\beta_{\mk{1}}}-\delta_{\mk{1}}^{\beta_0}- \delta_{\mk{1}}^{\beta_0}\ln(\delta_{\mk{1}})(\beta_{\mk{1}} - \beta_0)} \\
		&\qquad+ \abs{\delta_{\mk{1}}^{\beta_0}\ln(\delta_{\mk{1}})(\beta_{\mk{1}} - \beta_0) - \langle \delta_{\mk{1}}^{\beta_0}\ln(\delta_{\mk{1}})v, x-x_0 \rangle} \\
		&\qquad+ \abs{ \langle \delta_{\mk{1}}^{\beta_0}\ln(\delta_{\mk{1}})v, x-x_0 \rangle-\langle \delta_0^{\beta_0}\ln(\delta_0)v, x-x_0 \rangle}\\ 
		&\leq  \abs{\delta_{\mk{1}}^{\beta_{\mk{1}}}-\delta_{\mk{1}}^{\beta_0}- \delta_{\mk{1}}^{\beta_0}\ln(\delta_{\mk{1}})(\beta_{\mk{1}} - \beta_0)} \\
		&\qquad+ \abs{\beta_{\mk{1}} - \beta_0 - \langle v, x-x_0 \rangle}\delta_{\mk{1}}^{\beta_0}\abs{\ln(\delta_{\mk{1}})} \\
		&\qquad+ \abs{ \delta_{\mk{1}}^{\beta_0}\ln(\delta_{\mk{1}}) -\delta_0^{\beta_0}\ln(\delta_0)}\abs{v}\abs{x-x_0}.
	\end{aligned}
	}
	\mk{Recall that $\beta_1, \beta_0 \in [\beta_{min}, \beta_{max}]$. By Taylor expansion, there exist $\beta_2 \in (\beta_{min},\beta_{max})$ lying between $\beta_0$ and $\beta_1$, and $c(\beta_{min})$,} such that
	\formula{
		\abs{\mk{\delta_1^{\beta_1}-\delta_1^{\beta_0}-\delta_1^{\beta_0}\ln(\delta_1)(\beta_1-\beta_0)}} = \frac{1}{2}\abs{\mk{\delta_1^{\beta_2}\ln^2(\delta_1) (\beta_1 - \beta_0)^2}} \leq c\abs{\mk{\beta_1-\beta_0}}^2.
	}
	\mk{Since} $\beta$ is a Lipschitz continuous function on $\overline{D^*_{\mk{1}}}$, there exists a constant $c(X,D)$ such that
	\formula[eq:j1]{
		\abs{\delta_{\mk{1}}^{\beta_{\mk{1}}}-\delta_{\mk{1}}^{\beta_0}- \delta_{\mk{1}}^{\beta_0}\ln(\delta_{\mk{1}})(\beta_{\mk{1}} - \beta_0)} \leq c\abs{x-x_0}^2.
	}
	By Corollary \ref{cor:beta_holder}\mk{, $\beta$ is in $C^{\alpha + \epsilon}(\overline{D_{\mk{1}}^*})$ for some $\epsilon > 0$. Thus, by Lemma \ref{lem:lagrange}}, there exists a constant $c(X,D)$ such that for some $\epsilon > 0$ we have
	\formula[eq:j2]{
		\abs{\beta_{\mk{1}} - \beta_0 - \langle v, x-x_0 \rangle}\delta_{\mk{1}}^{\beta_0}\abs{\ln(\delta_{\mk{1}})} \leq c\abs{x-x_0}^{\alpha + \epsilon}.
	}
	By the fact that $\beta_0 > \alpha - 1$\mk{,} for some $\mk{\epsilon} > 0$ we have $\delta_{\mathbb{H}}(\cdot)^{\mk{\beta_0}}\ln(\delta_{\mathbb{H}}(\cdot)) \in C^{\alpha -1 + \mk{\epsilon}}(\overline{D^*_{\mk{1}}})$. Since $\abs{\nabla\beta}$ is bounded, there exists a constant $c(X,D)$ such that
	\formula[eq:j3]{
		\abs{ \delta_{\mk{1}}^{\beta_0}\ln(\delta_{\mk{1}}) -\delta_0^{\beta_0}\ln(\delta_0)}\abs{v}\abs{x-x_0} \leq c\abs{x-x_0}^{\alpha+\epsilon}
	}
	\mk{for some $\epsilon > 0$.} By combining \eqref{eq:f1alpha>1} with \eqref{eq:j1},\eqref{eq:j2},\eqref{eq:j3} we get the desired result for $\alpha \geq 1$.
	\end{proof}

\begin{lemma}
\label{lem:f_2}
There exists a constant $c(X,D)$ such that
\formula{
&\abs{f_{2}(x)} \leq c \, \max(\delta_D(x),\delta_{\mathbb{H}}(x))^{\beta_{min} - \mk{\eta}}\abs{\tilde{x}}^{{2\mk{\eta}}}\mathbbm{1}_{D_1 \cap \mathbb{H}}(x)
}
for some $c(D,X)>0$.
\end{lemma}

\begin{proof} 
Let $x \in D_1 \cap \mathbb{H}$. By Lemma~\ref{lem:odl^2}, \eqref{eq:oszacowanie1} and \eqref{eq:split} there exists a constant $c(X,D)$ such that
\formula{
\abs{f_{2}(x)}  \leq c \, \max(\delta_D(x),\delta_{\mathbb{H}}(x))^{\beta(x) - {\mk{\eta}}}\abs{\delta_D(x) - \delta_{\mathbb{H}}(x)}^{{\mk{\eta}}}\leq c \, \max(\delta_D(x),\delta_{\mathbb{H}}(x))^{\beta(x) - {\mk{\eta}}}\abs{\tilde{x}}^{{2\mk{\eta}}}.
}
Since $\max(\delta_D(x),\delta_{\mathbb{H}}(x)) \leq 1$ for $x \in D_{1} \cap \mathbb{H}$, we have 
\formula{
\max(\delta_D(x),\delta_{\mathbb{H}}(x))^{\beta(x) - {\mk{\eta}}} \leq \max(\delta_D(x),\delta_{\mathbb{H}}(x))^{\beta_{min} - {\mk{\eta}}}.
}
Thus we have
\formula{
\abs{f_{2}(x)} \leq c \, \max(\delta_D(x),\delta_{\mathbb{H}}(x))^{\beta_{min} - \mk{\eta}}\abs{\tilde{x}}^{{2\mk{\eta}}}.
}
\end{proof}

\begin{remark}
\label{rem:nabla_f_2}	
	By Remark~\ref{rem:k} and the fact that 
\formula{
	\beta_{max} \geq \max(\beta(u),\beta(-u)) \geq \frac{1}{2}\beta(u) + \frac{1}{2}\beta(-u) = \frac{\alpha}{2} \mk{,}
}
we have $2\mk{\eta} > 2\beta_{max} \geq \alpha$, thus in the case $\alpha \geq 1$, $\nabla f_2(x_0)$ exists and \mk{it} is equal to $0$.
\end{remark}

\begin{remark}
	\label{lem:f_3}
	We have 
	\formula{
		\abs{f_3(x)} \leq g(x)\mathbbm{1}_{\mk{\overline{D_1^*}} \setminus (D_1 \cap \mathbb{H})}+h(x)\mathbbm{1}_{\mathbb{H} \setminus D_1}(x) \mk{.}
	}
\end{remark}

\begin{lemma}
\label{lem:gen_est}
There exists $c_{gen}(X,D) > 0$ such that for $x \in D^*_{1/2}$.
\formula{
\abs{\mathcal{A}g(x)} \leq c_{gen}.
}
\end{lemma}

\begin{proof}
It is enough to show that $\abs{\mathcal{A}g(x_0)} \leq c_{gen}$ for some constant $c_{gen}\mk{(X, D)}$. \mk{The result in the general case follows then by an appropriate change of coordinates; see Remark~\ref{rem:set_coordinate_system}.} To \mk{estimate} $\mathcal{A}g(x_0)$\mk{,} we will compare the functions $g$ and $h$. Recall that, by Lemma~\ref{tw:harmonicznosc_w_R^d}, $h$ is harmonic on $\mathbb{H}$.

\smallskip

By \mk{\eqref{eq:split2}} we have
\formula{
\abs{\mathcal{A}(g - h)(x_0)} \leq \abs{\mathcal{A}f_1(x_0)} +\abs{\mathcal{A}f_2(x_0)} +\abs{\mathcal{A}f_3(x_0)}.
}
We claim that each summand is bounded by some $c(X,D) > 0$. Then $\mathcal{A}(g - h)(x_0)$ is well defined and $\abs{\mathcal{A}(g - h)(x_0)} \leq c\mk{(X, D)}$. By Lemma~\ref{tw:harmonicznosc_w_R^d}, $\mathcal{A}h(x_0)=0$, thus $\mathcal{A}g(x_0)$ is well defined and $\abs{\mathcal{A}g(x_0)} \leq c\mk{(X, D)}$\mk{, as desired}. 

To estimate $\abs{\mathcal{A}f_1(x_0)}$ we use Lemma~\ref{lem:f_1} and the fact that $f_1(x_0) = 0$. In \mk{the} case $\alpha < 1$ there exists $\epsilon(X,D) > 0$ and constants $c(X,D)$ such that
\formula{
\abs{\mathcal{A}f_1(x_0)} \leq \int_{\mathbb{R}^d}\abs{f_1(x)- f_1(x_0)} \nu(x - x_0) dx \leq c\int_{B(x_0,2)}\abs{x - x_0}^{\alpha + \epsilon} \nu(x - x_0) dx \leq c.
}
In \mk{the} case $\alpha > 1$\mk{, we have}
\formula{
\abs{\mathcal{A}f_1(x_0)} &\leq \int_{\mathbb{R}^d} \abs{f_1(x) - f_1(x_0) - \langle \nabla f_1(x_0), x-x_0  \rangle} \nu(x - x_0) dx \\
& \leq \int_{B(x_0,1/2)}\abs{x - x_0}^{\alpha + \epsilon} \nu(x - x_0) dx + \abs{\nabla f_1(x_0)}\int_{B(x_0,\delta_{D}(x_0))^c}\abs{ x-x_0 }\nu(x - x_0) dx .
}
\mk{Using additionally \eqref{rem:nabla_f1} with $\epsilon = \beta_{min} + 1 - \alpha$ (recall that $\beta_{min} > \alpha - 1$), we find that} there exist constants $c(X,D)$ such that
\formula{
\abs{\mathcal{A}f_1(x_0)} &\leq c+c\delta_{\mathbb{H}}(x_0)^{\beta(x_0)-\mk{\beta_{min} - 1 + \alpha}}\int_{B(x_0,\delta_{D}(x_0))^c}\abs{ x-x_0 }^{1-d-\alpha} dx \\
 & \leq c + c\delta_{D}(x_0)^{\beta(x_0)-\mk{\beta_{min}}} \leq \mk{c},
}
\mk{where in the last step we used the fact that $\beta(x_0) \geq \beta_{min}$}.

Finally, in the case $\alpha = 1$, we have
\formula{
\abs{\mathcal{A}f_1(x_0)}&\leq \langle \nabla f_1(x_0), b   \rangle+ \int_{\mathbb{R}^d} \abs{f_1(x) - f_1(x_0) - \langle \nabla f_1(x_0), x-x_0  \rangle \mathbbm{1}_{B(x_0,r)}(x)} \nu(x - x_0)dx,
}
where $b$ is the drift of the process $X$ and $r$ is an arbitrarily chosen radius defined in \eqref{eq:generator}. We set $r=2$. \mk{As in the case $\alpha > 1$, we find that} there exist constants $c(X,D)$ such that
\formula{
\abs{\mathcal{A}f_1(x_0)}&\leq c +  \int_{\mathbb{R}^d} \abs{f_1(x) - f_1(x_0) - \langle \nabla f_1(x_0), x-x_0  \rangle\mathbbm{1}_{B(x_0,2)}(x)} \nu(x - x_0) dx \\
&\leq c + c\int_{D_1 \cap \mathbb{H}}\abs{x - x_0}^{\alpha + \epsilon} \nu(x - x_0) dx \\ 
&\qquad+ \int_{(D_1 \cap \mathbb{H})^c \cap B(x_0,2)}\abs{\langle \nabla f_1(x_0), x-x_0  \rangle}\nu(x - x_0) dx \\ 
&\leq c + \abs{\nabla f_1(x_0)}\int_{B(x_0,\delta_{D}(x_0))^c \cap B(x_0,2)}\abs{ x-x_0 }\nu(x - x_0) dx
}
\mk{Again using \eqref{rem:nabla_f1} with $\epsilon = \beta_{min} / 2$, we obtain}
\formula{
\abs{\mathcal{A}f_1(x_0)}& \leq c + c\delta_{\mathbb{H}}(x_0)^{\beta(x_0)-\mk{\beta_{min}/2}}(\ln(2) - \ln(\delta_{\mathbb{H}}(x_0))\leq c \mk{,}
}
\mk{because $\beta(x_0) - \beta_{min}/2 \geq \beta_{min}/2$.}

\smallskip
By \eqref{eq:split} we have $f_2(x_0) = f_3(x_0) = 0$ and $\nabla f_3(x_0) = 0$. By Remark \ref{rem:nabla_f_2} for $\alpha \geq 1$ we have $\nabla f_2(x_0) = 0$, thus for every $\alpha \in (0,2)$ we have
\formula{
\mathcal{A}f_2(x_0) =  \int_{\mathbb{R}^d} f_2(x) \nu (x-x_0)dx \\
\mathcal{A}f_3(x_0) =  \int_{\mathbb{R}^d} f_3(x) \nu (x-x_0)dx.
}

\smallskip
To estimate $\abs{\mathcal{A}f_2(x_0)}$ we use Lemma~\ref{lem:odl^2}, Lemma~\ref{lem:f_2}, Lemma~\ref{lem:beta_ogr} and Remark~\ref{rem:k}. \mk{For appropriate constants $c(X,D)$,} we have
\formula{
&\abs{\mathcal{A}f_2(x_0)} \leq \int_{\mathbb{R}^d} \abs{f_2(x)}\nu(x-x_0)dx \\ & \leq c\int_{D \cap \mathbb{H} \cap B(x_0,2)}  \max(\delta_D(x),\delta_{\mathbb{H}}(x))^{\beta_{min} - \mk{\eta}}\abs{\tilde{x}}^{2\mk{\eta}} \nu(x - x_0) dx  \\
&\leq  c\int_{\mathbb{H} \cap B(x_0,2)} \delta_{\mathbb{H}}(x)^{\beta_{min} - \mk{\eta}}\abs{\tilde{x}}^{2\mk{\eta}} \abs{x - x_0}^{-d-\alpha} dx \\
&\leq  c\int_0^{2}  \int_{B^{(d-1)}(0,2)} t^{\beta_{min} - \mk{\eta}} \abs{\tilde{x}}^{2\mk{\eta}} (\abs{\tilde{x}}^2 + \abs{t-x_{0d}}^2)^{-\frac{d+\alpha}{2}} d\tilde{x}dt \\
&= c\int_0^{2} t^{\beta_{min} - \mk{\eta}} \int_0^{2} r^{d-2 + 2\mk{\eta}} (r^2 + \abs{t-x_{0d}}^2)^{-\frac{d+\alpha}{2}} dr dt.\\
}
Now we investigate the integral over $r$. For $b > 0$ we have
\formula{
\int_0^{2} r^{d-2 + 2\mk{\eta}} (r^2 + b^2)^{-\frac{d+\alpha}{2}} dr &= b^{2\mk{\eta} - \alpha -1} \int_0^{\frac{2}{b}} s^{d-2 + 2\mk{\eta}} (1 + s^2)^{-\frac{d+\alpha}{2}} ds \\
&\leq b^{2\mk{\eta} - \alpha-1} \int_0^{\frac{2}{b}}(1 + s^2)^{\frac{2\mk{\eta} - \alpha-2}{2}} ds \leq c b^{2\mk{\eta} - \alpha -1},
}
because, by Remark \ref{rem:k}, $2\mk{\eta} - \alpha -2 < -1$. Next we take $b = \abs{t-x_{0d}}$ and we have
\formula{
&c\int_0^{2} t^{\beta_{min} - \mk{\eta}}  \int_0^{2} r^{d-2 + 2\mk{\eta}} (r^2 + \abs{t-x_{0d}}^2)^{-\frac{d+\alpha}{2}} dr dt\\
&\leq  c \int_0^{2} t^{\beta_{min} - \mk{\eta}} \abs{t-x_{0d}}^{2\mk{\eta} - \alpha-1} dt\\
&= c x_{0d}^{\beta_{min}+\mk{\eta}-\alpha} \int_0^{\frac{2}{x_{0d}}} u^{\beta_{min} - \mk{\eta}} \abs{u-1}^{2\mk{\eta} - \alpha-1} du \\
&=  c x_{0d}^{\beta_{min}+ \mk{\eta} -\alpha} \expr{\int_0^2 u^{\beta_{min} - \mk{\eta}} \abs{u-1}^{2\mk{\eta} - \alpha-1} du + c\int_2^{\frac{2}{x_{0d}}} u^{\beta_{min} +\mk{\eta} -\alpha - 1}du} \leq \mk{c} ,
}
because, by Remark \ref{rem:k}, $\beta_{min}+ \mk{\eta} -\alpha > 0$, $\beta_{min} - \mk{\eta} > -1$ and $2\mk{\eta} - \alpha-1 > -1$.

\smallskip
To estimate $\abs{\mathcal{A}f_3(x_0)}$ we \mk{denote $A:= \{x \in \mathbb{R}^d: \abs{\tilde{x}} \leq 1, \abs{x_d} < \tfrac{1}{2} \abs{\tilde{x}}^2 \}$. By Lemma~\ref{lem:odl^2},} $(D \setminus \mathbb{H} \cup \mathbb{H} \setminus D) \cap B(x_0,1) \subset A$. We write
\formula{
\abs{\mathcal{A}f_3(x_0)} &\leq \int_{\mathbb{R}^d} \abs{f_3(x)}\nu(x-x_0)dx \\
&\leq c\int_A  (g(x) + h(x)) \nu(x-x_0)dx + \int_{B(x_0,1)^c}  (g(x) + h(x)) \nu(x - x_0) dx \\
&:= J_1 + J_2.
} 
For $x \in A$ we have $\frac{\mk{\abs{x_d}}}{\abs{\tilde{x}}} \leq \mk{\tfrac{1}{2}} \abs{\tilde{x}} < 1$\mk{,} and \mk{hence}
\formula{
\frac{\abs{x-x_0}}{\abs{x}} = \frac{\sqrt{\abs{\tilde{x}}^2+\abs{x_{0d}-x_{d}}^2}}{\sqrt{\abs{\tilde{x}}^2 + x_d^2}}=
\frac{\sqrt{1+\frac{\abs{x_{0d}-x_{d}}^2}{\abs{\tilde{x}}^2}}}{\sqrt{1+\frac{x_d^2}{\abs{\tilde{x}}^2}}} \geq \frac{1}{\sqrt{2}},
}
thus, by Assumption \ref{as:X} we have 
\formula{
	\nu(x-x_0) \leq c \nu(x).
} 
Moreover for $x \in A$ we have $\delta_{\mathbb{H}}(x),\delta_D(x) \leq c\abs{\tilde{x}}^2$ and \mk{hence}
\formula{
g(x) + h(x) \leq c\abs{\tilde{x}}^{2\beta_{min}} .
} 
Thus
\formula{
J_1 &\leq c\int_{B^{d-1}(0,1)} \abs{\tilde{x}^{2\beta_{min}}} \int_{-\abs{\tilde{x}}^2}^{\abs{\tilde{x}}^2}\frac{1}{(\abs{x_d}^2 + \abs{\tilde{x}}^2)^{\frac{d+\alpha}{2}}}d{x_d}d\tilde{x} \\
&\leq c\int_{B^{d-1}(0,1)} \abs{\tilde{x}}^{-d-\alpha +2\beta_{min} +2} d\tilde{x} \leq c, 
}
because $\beta_{min} > \alpha -1$. By Remark \ref{rem:g_Hol} we have
\formula{
\begin{aligned}
J_2 &\leq  \int_{B(x_0,1)^c} \nu(x - x_0) dx+ \int_{\mk{B(x_0,1)^c \cap \mathbb{H}}} x_d^{\beta (x_0)} \nu(x - x_0) dx \\
& \leq  c  + \int_{\mk{B(x_0,1)^c \cap \mathbb{H}}} x_d^{\beta (x_0)} \nu(x - x_0) dx.
\end{aligned} 
} 
Since for $x$ such that $\abs{x - x_0} \geq 1 $ we have $ x_d \leq \abs{x_0} + \abs{x - x_0} \leq 1/2 + \abs{x - x_0} \leq  2\abs{x - x_0}$, we have
\formula{
J_2 \leq  c  + c\int_{B(x_0,1)^c} \abs{x - x_0}^{\beta_{max} } \nu(x - x_0) dx  = c.
}

\mk{We have thus proved that all three summands: $\abs{\mathcal{A}f_1(x_0)}$, $\abs{\mathcal{A}f_2(x_0)}$ and $\abs{\mathcal{A}f_3(x_0)}$, are bounded by a constant $c(X,D)$. This completes the proof.}
\end{proof}

We keep the notation $c_{gen}$ till the end of this article.

\mk{We recall the following fundamental result on existence of boundary limits of ratios of harmonic functions.}

\begin{theorem}[{Theorem~2 and Example 1 in \cite{bib:jk16}}]
\label{thm:limit}
Let $D$ be open set, $z \in \partial D$. Suppose that non-negative functions \mk{$f_1$ and $f_2$} are regular harmonic functions in $D \cap B(z,r)$ and are equal to zero in $B(z, r) \setminus D$ for $r < R$. Then either one of \mk{$f_1$ and $f_2$} is zero everywhere in $D \cap B(z, r)$, or the finite, positive boundary limit of \mk{$f_1(x) / f_2(x)$} exists as $x \to z$, $x \in D$. 
\end{theorem}

\mk{Following \cite{bib:jk16}, we introduce the following notation}.

\begin{definition}
\mk{The} \emph{relative oscillation} of a function $f$ on \mk{the set} $D_r$ is given by \mk{the} formula
\formula{
RO_r \expr{f} = \frac{\sup_{x \in D_r}f(x)}{\inf_{x \in D_r}f(x)}.
}
\end{definition}

\mk{Note that if $f_1$ and $f_2$ are positive in $D_r$ for some $r$, then the} existence of the limit of $\mk{f_1 / f_2}$ as $x \to 0$ is equivalent to the condition $RO_r\mk{(f_1 / f_2)} \to 1$ as $r \to 0^+$.

\begin{definition}
We define the \emph{harmonic reduction} $g_r$ of the function $g$ by the formula
\formula{
g_r(x) = \mathbb{E}^x (g(X_{\tau_{D_r}})).
}
\end{definition}

\begin{lemma}
\label{lem:main}
For every $\epsilon >0$ there exists a radius $\mk{r_0}$ such that
\formula[eq:comp_g_time]{
RO_{\mk{r}}\expr{\frac{g_{\mk{r}}}{g}} \leq \frac{1+\epsilon}{1-\epsilon}
}
\mk{for every $0 < r \leq r_0$.}
\end{lemma}

\begin{proof}
Let $\mk{\phi}$ be a non-negative smooth function such that $\mk{\phi}(y) = 0$ for $\abs{y} > 1/2$ and $\int_{\mathbb{R}^d} \mk{\phi}(y) dy = 1$. For $k \geq 1$ we define $\mk{\phi}_k(y) =k^d\mk{\phi}(ky)$ and
\formula{
\mk{g_k}(x) := (\mk{\phi}_k \ast g)(x) := \int_{\mathbb{R}^d} \mk{\phi}_k(y) g(x - y)dy
}
and for $\mk{r} \leq 1/4$ let $D_{\mk{r}}^k := \{ y \in D_{\mk{r}} : \mk{\delta_D}(y) \geq 1/k \} $. Since $\mk{g_k}$ is a smooth function, $\mathcal{A}\mk{g_k}$ is well-defined everywhere. 

Let $x \in D_{\mk{r}}^k$ and $z \in B(0,\frac{1}{2k})$. By Lemma \ref{lem:gen_est} we have
$-c_{gen}  \leq \mathcal{A}g(x-z) \leq c_{gen}$. We claim that \mk{$\mathcal{A} g_k(x) = \phi_k * \mathcal A g(x)$ and consequently, by Lemma~\ref{lem:gen_est},}
\formula[eq:est_of_convolution]{
-c_{gen} \leq \mathcal{A}\mk{g_k}(x) \leq c_{gen}.
}

By Remark \ref{rem:g_Hol}\mk{, $g \in C^{\alpha + \epsilon}(D_1^*)$ for some $\epsilon > 0$. Hence, by} \eqref{eq:generator}, for $\alpha < 1$ we have
\formula{
\begin{aligned}
& \int_{\mathbb{R}^d}\int_{\mathbb{R}^d} \mk{\phi}_k(z) \abs{(g(y-z) - g(x-z))} \nu(y-x)dydz \\
\leq&  \int_{\mathbb{R}^d}\int_{\mathbb{R}^d}\mk{\phi}_k(z) \bigg( c\abs{y-x}^{\alpha+\epsilon} \mathbbm{1}_{B(x,1/4k)}(y) + 2\mathbbm{1}_{{B}^c(x,1/4k)}(y) \bigg) \nu(y-x)dydz < \infty.
\end{aligned}
}
Now, by Fubini Theorem, we have
\formula[eq:fubini1]{
\begin{aligned}
	\mathcal{A}\mk{g_k}(x) &= \int_{\mathbb{R}^d} (\mk{g_k}(y) -\mk{g_k}(x))\nu(y-x)dy  \\
	&= \int_{\mathbb{R}^d} \expr{ \int_{\mathbb{R}^d} \mk{\phi}_k(z) (g(y-z) - g(x-z))dz }\nu(y-x)dy \\
	&=\int_{\abs{z} < 1/\mk{2}k} \mk{\phi}_k(z) \expr{\int_{\mathbb{R}^d}(g(y-z) - g(x-z)) \nu(y-x)dy} dz \\
	& = \int_{\abs{z} < 1/\mk{2}k} \mk{\phi}_k(z) \mk{\mathcal{A} g(x - z)} dz = \mk{\phi_k * \mathcal{A} g(x)}.
\end{aligned}
}

In case $\alpha > 1$, by \mk{Remark \ref{rem:g_Hol} and} Lemma \ref{lem:lagrange} we write
\formula{
	& \int_{\mathbb{R}^d}\int_{\mathbb{R}^d} \mk{\phi}_k(z)\abs{(g(y-z) - g(x-z) - \langle \nabla g(x-z),y-x\rangle)}\nu(y-x)dydz  \\
	\leq& \int_{\mathbb{R}^d}\int_{\mathbb{R}^d} \mk{\phi}_k(z) \bigg(  c\abs{y-x}^{\alpha+\epsilon}\mathbbm{1}_{B(x,1/4k)}(y) \\
	&+ c(2 +\abs{y})\mathbbm{1}_{{B}^c(x,1/4k)}(y) \bigg) \nu(y-x)dydz < \infty.
}
\mk{Furthermore, $\nabla g_k = (\nabla g) * \phi_k$ in $D_r^k$.} Now, similarly as in \mk{the} case $\alpha < 1$, we write
\formula{
	\begin{aligned}
		\mathcal{A}\mk{g_k}(x) &= \int_{\mathbb{R}^d} (\mk{g_k}(y) -\mk{g_k}(x) - \langle \nabla \mk{g_k}(x),y-x\rangle))\nu(y-x)dy  \\
		=& \int_{\mathbb{R}^d} \expr{ \int_{\mathbb{R}^d} \mk{\phi}_k(z) (g(y-z) - g(x-z) - \langle \nabla g(x-z),y-x\rangle)dz }\nu(y-x)dy \\
		=&\int_{\abs{z} < 1/\mk{2}k} \mk{\phi}_k(z) \expr{\int_{\mathbb{R}^d}(g(y-z) - g(x-z) - \langle \nabla g(x),y-x\rangle) \nu(y-x)dy} dz \\
		 =& \int_{\abs{z} < 1/\mk{2}k} \mk{\phi}_k(z)\mk{\mathcal{A} g(x - z)} dz = \mk{\phi_k * \mathcal{A} g(x)}.
	\end{aligned}
}

Finally, in case $\alpha = 1$ we write \mk{similarly}
\formula{
	& \int_{\mathbb{R}^d}\int_{\mathbb{R}^d} \mk{\phi}_k(z)\abs{(g(y-z) - g(x-z) - \langle \nabla g(x-z),y-x\rangle \mathbbm{1}_{B(x,1/4k)}(y-z))}\nu(y-x)dydz  \\
	\leq& \int_{\mathbb{R}^d}\int_{\mathbb{R}^d} \mk{\phi}_k(z) \bigg(  c\abs{y-x}^{\alpha+\epsilon}\mathbbm{1}_{B(x,1/4k)}(y)+ \mk{2} \, \mathbbm{1}_{{B}^c(x,1/4k)}(y) \bigg) \nu(y-x)dydz < \infty .
}
Since $g \in C^{\alpha + \epsilon}(D^*_{1})$, \mk{$\nabla g$ is a continuus function in $D_1^*$, and $\nabla g_k = (\nabla g) * \ph_k$ in $D_r^k$. Thus,}
\formula{
\mathcal{A}\mk{g_k}(x) & = \mk{\langle \gamma, \nabla g_k(x) \rangle} + \int_{\mathbb{R}^d} (\mk{g_k}(y) -\mk{g_k}(x) - \langle \nabla \mk{g_k}(x),y-x\rangle \mathbbm{1}_{B(x,1/4k)}(y)))\nu(y-x)dy \\
&= \int_{\mathbb{R}^d} \langle \gamma, \nabla g(x - z) \rangle \phi_k(z) dz + \int_{\mathbb{R}^d} \biggl(\int_{\mathbb{R}^d} \mk{\phi}_k(z) (g(y-z) - g(x-z) \\
& \qquad - \langle \nabla g(x-z),y-x\rangle \mathbbm{1}_{B(x,1/4k)}(y-z))dz \biggr)\nu(y-x)dy \\
&=\int_{\abs{z} < 1/\mk{2}k} \mk{\phi}_k(z) \biggl(\langle \gamma, \nabla g(x - z) \rangle + \int_{\mathbb{R}^d}(g(y-z) - g(x-z) \\
 & \qquad - \langle \nabla g(x-z),y-x\rangle \mathbbm{1}_{B(x,1/4k)}(y)) \nu(y-x)dy \biggr) dz \\
& = \int_{\abs{z} < 1/\mk{2}k} \mk{\phi}_k(z) \mk{\mathcal{A} g(x - z)} dz = \mk{\phi_k * \mathcal{A} g(x)}.
}
\mk{This completes the proof of our claim~\eqref{eq:est_of_convolution}.}

\mk{Recall that} $\mk{g_k}$ is in $C_c^{\infty}(\mathbb{R}^d)$ and $\mathcal{A}$ restricted to $C_c^{\infty}$ coincides with the infinitesimal generator $\mathcal{L_D}$ of the process $X$. \mk{Denote} $\sigma(r,k) = \tau_{D_r^k}$. For $k \geq l$, by Dynkin's formula\mk{,} for $x \in D_r^l$ we have 
\formula{
\mathbb{E}^x \int_0^{\sigma(r,l)} \mathcal{A}\mk{g_k}(X_t)dt = \mathbb{E}^x (\mk{g_k}(X_{\sigma(r,l)})) - \mk{g_k} (x) .
}
Using \eqref{eq:est_of_convolution}, we get
\formula{
 -c_{gen} \mathbb{E}^x \sigma(r,l) \leq \mathbb{E}^x (\mk{g_k}(X_{\sigma(r,l)})) - \mk{g_k} (x) \leq c_{gen} \mathbb{E}^x \sigma(r,l) .
}
\mk{As $k \to \infty$, $g_k$ remains bounded by $1$ and it converges pointwise to $g$. Thus,}
\formula{
 -c_{gen} \mathbb{E}^x \sigma(r,l) \leq \mathbb{E}^x (g(X_{\sigma(r,l)})) - g(x) \leq c_{gen} \mathbb{E}^x \sigma(r,l) .
}
\mk{Now we pass to the limit as} $l \to \infty$. Since $\sigma(r,l)$ is \mk{an increasing} function of $l$, $g$ is bounded and $g(X_{\sigma(r,l)})$ converges almost surely to $g(X_{\tau_{D_r}})$\mk{,} we obtain that for all $x \in D_r$\mk{,}
\formula[eq:1st_ineq]{
 -c_{gen} \mathbb{E}^x ({\tau_{D_r}}) \leq \mathbb{E}^x (g(X_{\tau_{D_r}})) - g(x) \leq c_{gen} \mathbb{E}^x ({\tau_{D_r}}) .
}

In order to proceed we need to show a technical result comparing $g_r(x)$ with $\mathbb{E}^x({\tau_{D_r}})$. Since $\nu(y-z) \geq c\abs{y}^{-d-\alpha}$ for every $z \in D_r$ and $y \in {B}^c(0,r)$ we have
\formula[eq:green]{
\begin{aligned}
g_r(x) &= \int_{D^*_1 \setminus D_r} \expr{\int_{D_r} G_{D_r}(x,z)\nu(z-y)dz} g (y) dy \\
&\geq c\mathbb{E}^x({\tau_{D_r}})\int_{D^*_1 \setminus D_r} c\abs{y}^{-d-\alpha}g (y) dy,
\end{aligned}
}
where $G_{D_r}(x,z)$ is \mk{the} Green function of \mk{the} set $D_r$.

For $y \in B(0,1)$ with property that $2\abs{\tilde{y}} < y_d$ we have
\formula{
\abs{y} = \sqrt{y_d^2 + \abs{\tilde{y}}} \leq \sqrt{y_d^2 + \frac{1}{4}y_d^2} \leq c y_d.
}
By \eqref{eq:odl^2_2} we have
\formula{
\delta_D(y) &\geq y_d - \abs{\tilde{y}}^2 \geq y_d - \abs{\tilde{y}}  \geq \frac{y_d}{2} \geq c\abs{y},
}
so that 
\formula{
g(\mk{y}) \geq c\abs{y}^{\beta_{max}} \mk{.}
}
Thus, by using polar coordinates, we have
\formula{
\int_{D^*_1 \setminus D_r} \abs{y}^{-d-\alpha}g (y) dy &\geq  \int_{D_1 \setminus D_r} \abs{y}^{-d-\alpha}g (y) dy \\
& \geq c \int_{\{(\tilde{y},y_d):2\abs{\tilde{y}} < y_d, r < \abs{y} < 1 \}} \abs{y}^{-d-\alpha}\abs{y}^{\beta_{max}} dy \\ 
&\geq  c\int_{r}^{1} u^{-d-\alpha}u^{\beta_{max}}u^{d-1}du = c(r^{\beta_{max}-\alpha} - 1).
}
By combining this and \eqref{eq:green} we get that
\formula[eq:harm>time]{
g_r(x) \geq c(r^{\beta_{max}-\alpha} - 1)\mathbb{E}^x({\tau_{D_r}})
}
for $x \in D_r$.

Now \mk{suppose that} $\epsilon > 0$. By \eqref{eq:harm>time} there exists $\mk{r_0}$ such that  
\formula[eq:harm>time2]{
c_{gen} \mathbb{E}^x({\tau_{D_r}}) \leq \epsilon g_r(x)  
}
for $x\in B(0,r)$ and $r \le r_0$.
By combining \eqref{eq:harm>time2} with \eqref{eq:1st_ineq}\mk{,} we get that 
\formula[eq:g>harm]{
1-\epsilon \leq \frac{g_r(x)}{g(x)} \leq 1+\epsilon
}
for $x\in D_r$ and $r \le r_0$, which implies \eqref{eq:comp_g_time}.
\end{proof}

\begin{theorem}
Let $f$ be a non-negative function which is regular harmonic in $D_1$ and \mk{which} vanishes on $D^c \cap B(0,1)$. Then either $f$ is zero everywhere in $D$, or 
\formula{
\lim \frac{f(x)} {\delta_D(x)^{\beta(x)}} > \mk{0} \text{ exists as }x \to 0, x \in D.
}
\end{theorem}

\begin{proof}
Let $\epsilon >0$ and \mk{let $r_0$ be chosen according to Lemma~\ref{lem:main}}. By Theorem \ref{thm:limit} and \mk{the} fact that $g_{r_0}$ and $f$ are harmonic in $D_{r_0}$\mk{,} there exists radius $r \leq r_0$ such that
\formula{
RO_r\expr{\frac{f}{g_{r_0}}} \leq 1 + \epsilon.
}
For any \mk{positive} functions $f_1,f_2,f_3$  we have
\formula{
RO_r \expr{\frac{f_1}{f_2}} = \frac{\sup_{x \in D_r}\frac{f_1(x)}{f_2(x)}}{\inf_{x \in D_r}\frac{f_1(x)}{f_2(x)}} 
\leq \frac{\sup_{x \in D_r}\frac{f_1(x)}{f_3(x)}\sup_{x \in D_r}\frac{f_3(x)}{f_2(x)}}{\inf_{x \in D_r}\frac{f_1(x)}{f_3(x)}\inf_{x \in D_r}\frac{f_3(x)}{f_2(x)}} = RO_r \expr{\frac{f_1}{f_3}} RO_r \expr{\frac{f_3}{f_2}} .
}
Thus\mk{, by Lemma~\ref{lem:main},}
\formula{
RO_r\expr{\frac{f}{g}} \leq RO_r\expr{\frac{f}{g_{r_0}}} RO_{r}\expr{\frac{g_{r_0}}{g}} \leq\frac{(1+\epsilon)^2}{1 - \epsilon}.
}
\mk{Since} $\epsilon$ was chosen arbitrarily\mk{,} we have $RO_r \expr{\frac{f}{g}} \to 1$ as $r \to 0$.
\end{proof}

\begin{remark}
Note that, unlike in work of T. Grzywny, K.-Y. Kim and P. Kim from 2015 (\hspace{1sp}\cite{bib:gkk15}), with our methods we can not relax assumption of $D \in C^{1,1}$. If $D \in C^{1,\beta}$ for $\beta < 1$, then function $n(x)$ (and so $\beta(x)$ and $g_z(x)$) is not even a continuous function.
\end{remark}

\textbf{Acknowledgments} I would like to thank Mateusz Kwaśnicki for his help, valuable comments and ideas that helped to improve our results. I would also like to thank Zhen-Qing Chen for his valuable comments and discussion.

%
%                            ---------- o ----------
%

\bibliographystyle{plain}
\bibliography{mybib}

\end{document}